\documentclass[11pt]{amsart}
\usepackage[utf8]{inputenc}

\usepackage{amssymb,amsthm,amsfonts,enumerate, epsfig,epstopdf,caption, graphicx, xcolor}
\usepackage{amsmath}
\usepackage{subcaption}
\usepackage{hyperref}

\title{Linear Bounds of the Crosscap Number of Knots} 
\author{Rob McConkey}

\begin{document}

\maketitle

\newtheorem{innercustomgeneric}{\customgenericname}
\providecommand{\customgenericname}{}
\newcommand{\newcustomtheorem}[2]{%
  \newenvironment{#1}[1]
  {%
   \renewcommand\customgenericname{#2}%
   \renewcommand\theinnercustomgeneric{##1}%
   \innercustomgeneric
  } {\endinnercustomgeneric} }

\newcustomtheorem{customthm}{Theorem}
\newcustomtheorem{customlemma}{Lemma}
\newcustomtheorem{customprop}{Proposition}
\newcustomtheorem{customconjecture}{Conjecture}
\newcustomtheorem{customcor}{Corollary}
\newcommand{\Q}{{\mathbb{Q}}}
\newcommand{\R}{{\mathbb{R}}}
\newcommand{\Z}{{\mathbb{Z}}}
\newcommand{\N}{{\mathbb{N}}}
\newcommand{\C}{{\mathbb{C}}}

\theoremstyle{plain}
\newtheorem*{ack*}{Acknowledgements}
\newtheorem{theorem}{Theorem}[section]
\newtheorem{lemma}[theorem]{Lemma}
\newtheorem{prop}[theorem]{Proposition}
\newtheorem{cor}[theorem]{Corollary}
\newtheorem{predefinition}[theorem]{Definition}
\newtheorem{conjecture}[theorem]{Conjecture}
\newtheorem{preremark}[theorem]{Remark}
\newenvironment{remark}%
 {\begin{preremark}\upshape}{\end{preremark}} \newenvironment{definition}%
  {\begin{predefinition}\upshape}{\end{predefinition}}

\newtheorem{ex}[theorem]{Example}
\newtheorem{ques}[theorem]{Question}

\makeatletter
\newtheorem*{rep@theorem}{\rep@title}
\newcommand{\newreptheorem}[2]{%
\newenvironment{rep#1}[1]{%
 \def\rep@title{#2 \ref{##1}}%
 \begin{rep@theorem}}%
 {\end{rep@theorem}}}
\makeatother

 \bibliographystyle{plain}
\newreptheorem{theorem}{Theorem}
\begin{abstract}
    Kalfagianni and Lee found two-sided bounds for the crosscap number of an alternating link in terms of certain coefficients of the Jones polynomial. We show here that we can find similar two-sided bounds for the crosscap number of Conway sums of strongly alternating tangles. Then we find families of links for which these coefficients of the Jones polynomial and the crosscap number grow independently. These families will enable us to show that neither linear bound generalizes for all links. 
\end{abstract}

\section{Introduction}
In \cite{LEE} Kalfagianni and Lee show that the crosscap number of an alternating link admits two-sided linear bounds in terms of certain coefficients of the Jones polynomial of the link. The purpose of this paper is twofold; we first generalize the result of \cite{LEE} for links that are the Conway sum of strongly alternating tangles. Second we construct families of knots obstructing the generalization of the result of \cite{LEE} to arbitrary knots. \\
\indent
For a link $L$ let
$$V(L) = \alpha_L t^n + \beta_L t^{n-1} + \dots + \beta_L' t^{m+1} + \alpha_L' t^m$$
be the Jones polynomial, and let $T_L = |\beta_L| + |\beta_L'|$.
\begin{definition}
    For a non-orientable connected surface, $S$, bounded by a link $L$, the \textit{crosscap number} is defined to be $$C(S) = 2 - \chi(S) - k_L,$$ where $k_L$ is the number of components of $L$. Then the \textit{crosscap number} for a link $L$, denoted $C(L)$ will be the minimum crosscap number of all non-orientable surfaces bounded by the link.
\end{definition}

 \indent
 In the first part of the paper we will find two sided linear bounds for $C(L)$, where $L$ is a Conway sum of tangles in terms of $T_L$. We also define many of the necessary terms for Theorem~\ref{cor:TLTwist} in that section, which we state here.

 \begin{theorem}
 \label{cor:TLTwist}
       Let $T_1$ and $T_2$ be non-splittable, twist reduced, strongly alternating tangles whose Conway sum is a link $L$. Let $C(L)$ be the crosscap number of $L$ and $k_L$ be the number of components of $L$. Then
 $$\left\lceil\frac{T_L}{6} \right\rceil - k_L \leq C(L) \leq 2T_L + k_L + 8.$$
 \end{theorem}
 A key ingredient in the proof of Theorem~\ref{cor:TLTwist}, is Theorem~\ref{thm:CrossMain} which we state below. Theorem~\ref{thm:CrossMain} gives us bounds for $C(L)$ in terms of the crosscap numbers of the closures of the tangles which sum to $L$.
 \begin{theorem}
 \label{thm:CrossMain}
     Let $T_1$ and $T_2$ be non-splittable, twist reduced, strongly alternating tangles, and let $L$ be the link formed by the Conway sum of $T_1$ and $T_2$. Let $K_{iN}$ and $K_{iD}$ be the link formed by the numerator closure and denominator closures of $T_i$ respectively, $i \in \{1,2\}$. We have 
$$m - 2 \leq C(L) \leq m + 2 $$
    where $m = \text{min\{} C(K_{1N}) + C(K_{2N}), C(K_{1D}) + C(K_{2D})\} $.
 \end{theorem}
 Having Theorem~\ref{thm:CrossMain} at hand we use a result of \cite{LEE} and the additivity of twist numbers for strongly alternating tangles to find bounds for $C(L)$ in terms of the twist number of $L$. Then using these bounds and a generalization of Theorem 1.6 from \cite{Twist} by Futer, Kalfagianni, and Purcell gives us Theorem~\ref{cor:TLTwist}.\\
 \indent
 In section 5 we show that for arbitrary knots the crosscap number and $T_L$ are independent. Specifically we show:
 \begin{theorem}
 \label{Thm:nongen}
     We have the following;
    \begin{enumerate}[(a)]
        \item There exists a family of links for which $T_L \leq 2$, but $C(L)$ is arbitrarily large.
        \item There exists a family of links for which $C(L) \leq 3$, but $T_L$ is arbitrarily large. 
    \end{enumerate}
 \end{theorem}
 To show part (a) of Theorem~\ref{Thm:nongen}, we use work by Teragaito~\cite{TorusCC} to find a family of torus knots $T(p,q)$ where $C(T(p,q))$ grows with $q$ and $q$ can be made arbitrarily large. On the other hand we will show that $T_{T(p,q)} \leq 2$. \\
 \indent
 For part (b) we will introduce a family of Whitehead doubles for which the crosscap number is always bounded by 3 but $|\beta_W'|$ can be made arbitrarily large. Work by Clark~\cite{Clark} shows that for all links $C(L) \leq 2g(L) + 1$ which shows that for all Whitehead doubles, $C(W) \leq 3$. On the other hand, using work by Stoimenow~\cite{Stoimenow}, we are able to compute $|\beta'_W|$ for B-adequate links. Then we find a family of B-adequate Whitehead doubles for which $|\beta_W'|$ can be made arbitrarily large. \\
 \indent 
 We note that all the links constructed in Theorem~\ref{Thm:nongen} are non-hyperbolic. This leaves the question of whether Theorem~\ref{cor:TLTwist} may be generalized to all hyperbolic links. See Section 6 for more detail. 
 
\begin{ack*}
    The author thanks their advisor Efstratia Kalfagianni for guidance, helpful discussions, and comments on earlier drafts. Part of this research was supported in the form of graduate Research Assistantships by NSF grants DMS-2004155 and DMS-2304033 and funding from the NSF/RTG grant DMS-2135960.
\end{ack*}

\section{Crosscap Bounds on Connection of Two Strongly Alternating Tangles}
In this section we will work to prove Theorem~\ref{thm:CrossMain}.
\subsection{Preliminaries and the Upper bound of Theorem~\ref{thm:CrossMain}}
 We start with a couple definitions.  
 \begin{definition}
     A \textit{tangle} is a graph in the plane contained within a box which intersects the box at the four corners with one-valent vertices, with all other vertices, contained inside the box, four-valent, and given over/under crossing data. We label the four, 1-valent vertices NW, NE, SE, SW, positioned according to Figure~\ref{fig:numden}.
 \end{definition}
 \begin{definition}
      The \textit{closure} of a tangle is the link which results when we connect the NW and NE points along the box and SW and SE points along the box as seen in the center panel of Figure~\ref{fig:numden}, this is called the \textit{numerator closure}. If we close as in the right hand panel of Figure~\ref{fig:numden} we call it the \textit{denominator closure}. A tangle is \textit{strongly alternating} if both closures are prime and alternating. 
      \begin{figure}[h]
          \centering
          \includegraphics[width=9cm]{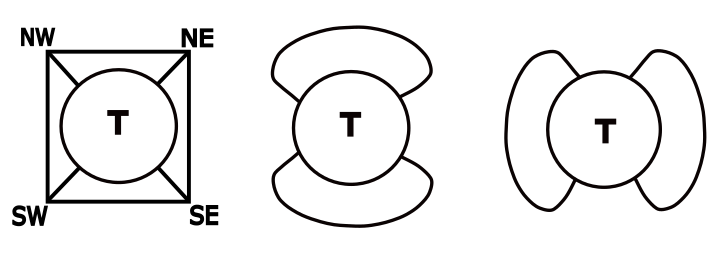}
          \caption{Left: A tangle inside a box with directional strands labelled, Center: numerator closure, Right: denominator closure.}
          \label{fig:numden}
      \end{figure}
 \end{definition}
 \begin{definition}
    A \textit{Conway sphere} is a 2-sphere which intersects a knot or link transversely in four points. A \textit{Conway sum} is a sum of tangles as shown in Figure~\ref{fig:LargeTangle}. For our purposes, a Conway sphere $\Sigma$ will be positioned such that it intersects a Conway sum at the four one valent vertices for one of the tangles in the sum. Notice if we let $S$ be a spanning surface for our Conway sum, then $S \cap \Sigma$ will contain two arcs and a possibly empty collection of simple closed curves.
\end{definition}
\begin{figure}[h]
    \center
    \includegraphics[width=8cm]{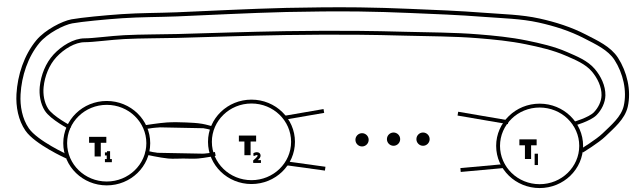}
    \caption{An example of a Conway sum of $l$ tangles.  }
    \label{fig:LargeTangle}
\end{figure}
\begin{figure}[h]
    \centering
    \includegraphics[width=5cm]{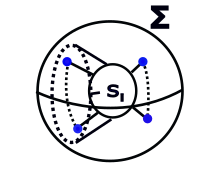}
    \caption{Here we see a tangle contained within a Conway Sphere. The blue dots represent the intersection of $T_i$ with $\Sigma$. Then the dotted lines are the intersections of $S$ with $\Sigma$.}
    \label{fig:ConwaySph}
\end{figure}
We are now ready to begin proving Theorem~\ref{thm:CrossMain}. We separate it into the upper and lower bounds, beginning with the upper bound: 

\begin{lemma}
\label{lem:upperbd}
Let $T_1$ and $T_2$ be a pair of non-splittable, strongly alternating tangles. Let $L$ be the link formed by the Conway sum of $T_1$ and $T_2$. Let $K_{iN}$ and $K_{iD}$ be the numerator and denominator closures, respectively, of $T_1$ and $T_2$. If $C(L)$ is the crosscap number of $L$, then
$$C(L) \leq \text{min\{} C(K_{1N}) + C(K_{2N}) + 2, C(K_{1D}) + C(K_{2D}) + 2\}.$$ 
\end{lemma}
\begin{proof}
Start with a pair of strongly alternating tangles, $T_1$ and $T_2$. Let $K_{1N}$ and $K_{2N}$ be the links acquired by the numerator closures of the tangles. Let $S_1$ and $S_2$ be non-orientable spanning surfaces which realize the crosscap numbers of $K_{1N}$ and $K_{2N}$ respectively. We find a spanning surface $S$ of $L$ by attaching $S_1$ and $S_2$ with a pair of bands bounded by the strands along which the Conway sum was taken. Notice that as we do not cut $S_1$ and $S_2$, $S$ will also be non-orientable.  \\
\indent
 Now we study the relationship between $C(S)$ and the sum of $C(S_1)$ and $C(S_2)$. We remind the reader that $C(S) = 2- \chi(S) - k $. The difference between $S$ and the disjoint union of $S_1$ and $S_2$ is the two connecting bands used to construct $S$. So, $\chi(S) =\chi(S_1) + \chi(S_2) - 2 $.  \\
\indent
Next we compare the number of link components in $L$ with the total in $K_{1N}$ and $K_{2N}$. The gluing of the East strands of $K_{1N}$ to the West strands of $K_{2N}$ will reduce the number of components by 1 as we are connecting two disjoint links. The other attachment can  increase or decrease the number of components by 1, or keep it the same. Then $k_L = k_{K_{1N}} + k_{K_{2N}} - \epsilon$ where $\epsilon = 0,1,\text{or }2$.\\
\indent 
Now we substitute for $\chi(S)$ and $k_L$ to find:
$$C(S) = 2 - \chi(S_1) - \chi(S_2) + 2  - k_{K_{1N}} - k_{K_{2N}} + \epsilon = C(K_{1N}) + C(K_{2N}) + \epsilon.$$
Hence: 
$$C(L) \leq C(S) =  C(K_{1N}) + C(K_{2N}) + 2$$ as $\epsilon = 2,$ will give the weakest upper bound. By the same argument with the denominator closures of $T_1$ and $T_2$, we find that $$C(L) \leq C(K_{1D}) + C(K_{2D}) + 2,$$ giving us the claim.

\end{proof}
Here we remark that Lemma~\ref{lem:upperbd} will hold even if we take two general tangles. Notice in the proof that we do not use the fact that $T_1$ or $T_2$ are non-splittable or strongly alternating. As this will not hold true for the other statements we included these hypotheses for uniformity. 
\subsection{Technical Lemmas}
Before we show the lower bound, we will need some more background, as well as some technical results. Lemma~\ref{lem:linkcomp} below was discussed in the proof of Lemma~\ref{lem:upperbd}.
\begin{lemma}
\label{lem:linkcomp}
    Let $L$ be the Conway sum of the tangles $T_1$ and $T_2$, and let $K_1$ and $K_2$ be closures of $T_1$ and $T_2$. If $k_L$ is the number of link components for $L$, and $k_1$ and $k_2$ the number of link components for $K_1$ and $K_2$, respectively, then $k_L = k_1 + k_2 - \epsilon$ for $\epsilon=0,1,2$.
\end{lemma}

\begin{definition}
Let $L$ be a link in $S^3$ and let $N(L)$ be a neighborhood of $L$. A spanning surface $S$ of $L$ in $S^3$ is defined to be \textit{meridianally boundary compressible} if there exists a disk $D$ embedded in $S^3/N(L)$, such that $\partial D = \alpha \cup \beta$ where $\alpha = D \cap \partial N(L)$ and $\beta = D \cap S$. Notice both $\alpha$ and $\beta$ are arcs, $\beta$ does not cut off a disk of $S$, $\partial D \cap \partial S$ cuts $\partial S$ into two arcs $\phi_1, \phi_2$ and $\alpha \cup \phi_i$ is a meridian of the link for one of $i = 1,2$ as shown in Figure~\ref{fig:meridian}. A spanning surface is said to be \textit{meridianally boundary incompressible} if no such disk exists. 
\begin{figure}[h]
    \center
    \includegraphics[width=8cm]{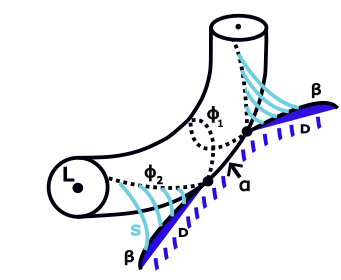}
    \caption{This figure shows a case where $L$ is meridianally boundary compressible. As $\phi_1 \cup \alpha$ create a meridian of $N(L)$.   }
    \label{fig:meridian}
\end{figure}
 \end{definition}
\begin{definition}
 Given an alternating projection of a link $L$ on $S^2$, we modify it so that in a neighborhood of each crossing, we have a ball whose equator lies on $S^2$ such that the over strand runs over the ball and the under goes underneath, see Figure~\ref{fig:menasco} for reference. We call every such ball a \textit{Menasco ball}, and we call such an embedding of $L$ relative to $S^2$ a \textit{Menasco projection} $P$ with $n$ crossings.\\
\end{definition}
 \begin{figure}[h]
     \centering
     \includegraphics[width=9cm]{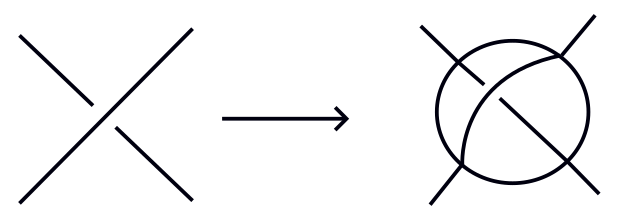}
     \caption{Here we put a crossing into Menasco form, with the over strand running across the top of the ball and the lower strand running along the bottom.}
     \label{fig:menasco}
 \end{figure}
 \begin{definition}
 We say that a surface $S$ intersects a Menasco ball $B_i$ in a \textit{crossing band} if $S \cap B_i$ consists of a disc bounded by the over and under strands on $\partial  B_i$ along with opposite arcs along the equator of $B_i$. We refer the reader to figures 28-30 in \cite{SpanAlt} for reference.
 
\indent
 Let $F = S^2/\bigcup_i B_i$, where the $B_i$ are the Menasco balls for $L$. Given an incompressible (not necessarily meridianally) surface $S$ spanning $L$, we can isotope $S$ so that  
 \begin{enumerate}[i.]
     \item $S \cap F$ is a collection of simple closed curves and arcs with endpoints on $L$ or the equator of a Menasco ball.
     \item $S$ is disjoint whereever possible from the interior of the $B_i$, including along $N(L)$. The only exception will be at crossing bands.
 \end{enumerate}
 We say such a surface isotoped in this way is in \textit{Menasco form}.
 \end{definition}

We will also need Lemma 5.1 from \cite{SpanAlt} by Adams and Kindred which we state here as Lemma~\ref{lem:crossingband}. Lemma~\ref{lem:crossingband} is required to prove Lemma~\ref{lem:AdamsKindred}, which is proven as Corollary 5.2 in \cite{SpanAlt}. Lemma~\ref{lem:AdamsKindred} is essential to our proof of the lower bound, as it guarantees the connectedness of any spanning surface of the numerator or denominator closures of tangles that we consider. 

\begin{lemma}
 \label{lem:crossingband}
 An incompressible and meridianally boundary incompressible surface $S$ spanning an alternating link $L$ can be isotoped relative to a given nontrivial Menasco projection $P$ to obtain a crossing band.
 \end{lemma}

\begin{lemma}
\label{lem:AdamsKindred}
Any spanning surface for a non-splittable, alternating link is connected.
\end{lemma}

 \begin{proof}
 This proof uses induction on the number of crossings a non-splittable, alternating link contains. As the unknot contains only one link component, any spanning surface for the unknot must be connected. Now consider a non-splittable alternating link $L$ which has $n$ crossings, and let $S$ be a spanning surface for $L$ then $S$ is either incompressible and meridianally boundary incompressible or a finite sequence of compressions take $S$ to an incompressible and meridianally boundary incompressible surface $S'$. Choose a reduced alternating diagram of $L$ and put $S'$ into Menasco form relative to $L$. By Lemma~\ref{lem:crossingband}, we can isotope $S'$ such that $S'$ contains a crossing band in at least one of the Menasco balls, $M$. Further, when we cut open the link along $M$, we find a spanning surface $S''$ for a non-splittable alternating link with fewer crossings $L'$. Part of the equator of $M$ replaces the crossing strands and guarantees that $S''$ is a spanning surface of $L'$. Then, by induction, as $S''$ is a spanning surface for a link of $n-1$ crossings, it is connected. Regluing in the crossing band does not disconnect our surface, showing that $S'$ and $S$ are connected. \end{proof}

\begin{lemma}

\label{lem:nonsep}
Let $\Sigma$ be a Conway sphere which intersects a Conway sum $L$ of two strongly alternating tangles $T_1$ and $T_2$ in $S^3$ such that $\Sigma$ separates $T_1$ and $T_2$. If we let $S$ be a spanning surface of $L$ and $S \cap \Sigma$ contains a simple closed curve $\gamma$ such that $\gamma$ does not separate the two arcs in $S \cap \Sigma$ on $\Sigma$, then there exists an isotopy on $S$ which will eliminate $\gamma$.
\end{lemma}
\begin{proof}
Assume there is only one closed curve $\gamma$ contained in $S \cap \Sigma$. First we consider the case where cutting $S$ along $\gamma$ and gluing in disks along the resulting boundary components results in two disconnected closed components. Notice that this implies that $S$ has a disconnected closed component, contradicting that $S$ is a crosscap realizing surface. \\
\indent
Next assume that cutting along $\gamma$ and gluing in disks does not result in a closed surface component. Then cutting $S$ along $\Sigma$ and gluing disks along the copies of $\gamma$ will result in spanning surfaces for a closure of $T_1$ and a closure of $T_2$, both of which are connected by Lemma~\ref{lem:AdamsKindred}. Reversing this procedure everywhere but $\gamma$ results in a new connected surface $S'$ which also spans $L$. But as $S'$ has two additional disks, $\chi(S') = \chi(S) + 2$, showing that $C(S') < C(S)$ contradicting that $S$ is a crosscap realizing spanning surface. 
\\
\indent
The only remaining possibility is if cutting $S$ along $\gamma$ and gluing disks to the two resulting boundaries, results in a single closed surface component, $U$. Then $U$ will separate $S^3$ into two disjoint spaces. As $\gamma$ does not separate the two arcs on $\Sigma$, one side of $U$ must not contain any part of $S$. But then we can move $U$ to the opposite side of $\Sigma$ and re-glue it to $S$ along $\gamma$ to find an isotopy of $S$ for which $\gamma$ is no longer in $\Sigma \cap S$.  \\
\indent
In the case that we have multiple such closed curves along $\Sigma$ we do the same as above starting with the innermost closed curve. The innermost closed curve in this case is the one that bounds an empty disk on $\Sigma$. Hence showing the claim.  
\end{proof}
By Lemma~\ref{lem:nonsep}, we can choose $S$ such that the only simple closed curves in $\Sigma \cap S$ are those which bound two disks each containing an arc. Next we show that $S$ can be chosen so that $S \cap \Sigma$ contains at most one such simple closed curve.

\begin{lemma}
\label{lem:curveslemma}
There exists a spanning surface $S$ for $L$, where $L$ is the Conway sum of two strongly alternating tangles, such that $C(S) = C(L)$ and $\Sigma \cap S$ contains at most one closed curve. 
\end{lemma}

\begin{proof}
By Lemma~\ref{lem:nonsep} we can assume that if $\Sigma \cap S$ contains closed curves $\gamma_1,...,\gamma_n$, they each split $\Sigma$ such that the two arcs lie on opposite disks. Assume we have $n>1$ such closed curves in $\Sigma \cap S$, and let $\gamma_1$ and $\gamma_2$ be such that $\gamma_1$ bounds a disk on $\Sigma$ such that no other $\gamma_i$ are in the disk and $\gamma_2$ bounds a disk where the only closed curve in it is $\gamma_1$. \\
\indent
We now find a spanning surface $S'$ such that $C(S') = C(L)$ and $S' \cap \Sigma$ contains $n-2$ closed curves. We start with $S$ and cut along $\gamma_1$ and $\gamma_2$ and then glue in annuli whose boundaries are a copy of $\gamma_1$ and a copy of $\gamma_2$. Then as the Euler characteristic of an Annulus is 0, this cutting and gluing operation will result in $\chi(S) = \chi(S')$. \\
\indent
It remains to show that $S'$ will be a connected surface. As in the proof of Lemma~\ref{lem:nonsep} we cut $S$ along $\Sigma$ and glue in disks along each $\gamma_i$ except for $\gamma_1$ and $\gamma_2$ which we glue a pair of annuli.  This results in spanning surfaces $S_1$ and $S_2$ for closures of $T_1$ and $T_2$ which by Lemma~\ref{lem:AdamsKindred} are connected. If we reverse this procedure everywhere except $\gamma_1$ and $\gamma_2$ the result will be $S'$ and as we only remove disks before regluing, $S'$ will be connected. Hence, we have found a spanning surface $S'$ for $L$ such that $C(S') = C(L)$ and $S' \cap \Sigma$ contains two less closed curves. Hence, repeating for all such pairs of closed curves in $S \cap \Sigma$ we will find the claim.  
\end{proof}
\begin{lemma}
\label{lem:closedsurfaces}
We can choose a surface $S$ which spans a link $L$ such that $C(S) = C(L)$ and cutting along $\Sigma$ will not give us a closed surface component. 
\end{lemma}
\begin{proof}
This was shown in the proof of Lemma~\ref{lem:nonsep}. 
\end{proof}

\subsection{Lower Bound of Theorem~\ref{thm:CrossMain}}
In this subsection we will prove the lower bound of Theorem~\ref{thm:CrossMain}, which will be restated as Lemma~\ref{lem:lowerbd}. We start by discussing what happens when we cut our link $L$ along $\Sigma$. Assume that $S$ is a non-orientable spanning surface for $L$ with $C(S) = C(L)$. The two arcs on $\Sigma$ will define how we close $T_1$ and $T_2$ after cutting. We let $K_1$ and $K_2$ be these closures. To see that $K_1$ and $K_2$ are the numerator or denominator closures consider a crossing which as a vertex in the tangle graph is adjacent to a 1-valent vertex. If the exterior regions for a tangle are the faces bounded by the box in the graph then one of the two exterior regions adjacent to the crossing must be included in $S$. This means the boundary of this region will result in the numerator or denominator closures.  \\
\indent
We are now ready to prove the lower bound of Theorem~\ref{thm:CrossMain}.

\begin{lemma}
\label{lem:lowerbd}
Let $T_1$ and $T_2$ be non-splittable, strongly alternating tangles and $L$ the link resulting from the Conway sum of $T_1$ and $T_2$. Also, let $S$ be a spanning surface of $L$ such that $C(L) = C(S)$.  Then:
$$ C(K_1) + C(K_2) - 2 \leq C(L).$$
\end{lemma}
\begin{proof}
Let $S$ be a non-orientable spanning surface for $L$ such that $C(L) = C(S)$. By Lemma~\ref{lem:curveslemma} $S$ can be chosen such that the intersection of $S$ with $\Sigma$ contains at most one closed curve $\gamma$ and that both disks $\gamma$ bounds contain arcs. We cut $S$ along $\Sigma$ and if $\gamma$ exists we glue a disk to each copy to get spanning surfaces $S_1$ and $S_2$ for $K_1$ and $K_2$ respectively. By Lemma~\ref{lem:closedsurfaces} we know that $S_1$ and $S_2$ will not have closed components and by lemma~\ref{lem:AdamsKindred} $S_1$ and $S_2$ must be connected as $K_1$ and $K_2$ are alternating. \\
\indent
If $k_1$ and $k_2$ are the number of link components for $K_1$ and $K_2$ respectively, then by Lemma~\ref{lem:linkcomp} $k_1 + k_2 - \epsilon = k_L$ where $\epsilon = 0,1,2$.\\
\indent
Next we consider how the Euler characteristics of $S$ and the sum of the Euler characteristics of $S_1$ and $S_2$ will be related. We know that $\Sigma \cap S$ contains two arcs and at most one closed curve by Lemma~\ref{lem:curveslemma}. Then cutting the two arcs along $\Sigma$ will increase the Euler characteristic by 2. Assume there are $n$ closed curves, gluing disks along the two copies after cutting will further increase the Euler characteristic by $2n$. The final consideration we have to make is whether $S_1$ and $S_2$ are non-orientable, let $t = 0,1,2$ be the number of $S_i$ which are orientable. Notice we will have to add $t$ half twist bands to make sure all the $S_i$ are non-orientable decreasing the Euler Characteristic by $t$. Now we see that $\chi(S) = \chi(S_1) + \chi(S_2) - 2 - 2n + t$. Then:
$$C(S_1) + C(S_2) = 4 - \chi(S) - 2 - 2n + t - k_L - \epsilon = C(L) - 2n + t - \epsilon.$$
Simplifying we find $$C(K_1) + C(K_2) + 2n - t + \epsilon \leq C(L).$$ This will be the weakest when $n = 0$, $t = 2$, and $\epsilon = 0$. Hence we find $C(K_1) + C(K_2) -2 \leq C(L)$.
\end{proof}
We restate Theorem~\ref{thm:CrossMain}:
\begin{reptheorem}{thm:CrossMain}
     Let $T_1$ and $T_2$ be non-splittable, twist reduced, strongly alternating tangles. Let $L$ be the link formed by the Conway sum of $T_1$ and $T_2$. Let $K_{iN}$ be the link formed by the numerator closure of $T_i$, $i \in \{1,2\}$, similarly $K_{iD}$ will be the link formed by the denominator closure. If we let $m = \text{min\{} C(K_{1N}) + C(K_{2N}), C(K_{1D}) + C(K_{2D})\} $ then,
$$C(K_1) + C(K_2) - 2 \leq C(L) \leq m + 2. $$

 \end{reptheorem}
Theorem~\ref{thm:CrossMain} follows directly from Lemma~\ref{lem:upperbd} and Lemma~\ref{lem:lowerbd}. A similar result exists to Theorem~\ref{thm:CrossMain} for the cross cap number of connected sums. In particular, Clark~\cite{Clark} showed with a strategy similar to our own, that if $K_1$ and $K_2$ are knots, then 
 $$C(K_1) + C(K_2) - 1 \leq C(K_1 \# K_2) \leq C(K_1) + C(K_2).$$

\section{Crosscap Number, Twist Number, and the Jones Polynomial}
\subsection{Twist Number Bounds}
Now we have a relationship between the crosscap numbers of the Conway sum of two tangles and the closures of the tangles which compose it. Unfortunately, the bounds depend upon the tangles and which closures we take. But we can use Theorem~\ref{thm:CrossMain} to find bounds for $C(L)$ entirely dependent upon $L$. Before proceeding with the statements, we will need a definition.
\begin{definition}
    
    The \textit{twist number} of a link diagram or a tangle diagram is the number of twist regions a link diagram contains, where a \textit{twist region} is a maximal collection of bigon regions contained end to end. We call a link diagram \textit{twist-reduced} if any simple closed curve which meets the link diagram transversely at four points, with two points adjacent to one crossing and the other two another crossing, bounds a possibly empty collection of bigons arranged end to end between the two crossings. 
\end{definition}
We take a brief pause to mention that we can take the Conway sum of more than two tangles. In particular, for tangles $T_1, T_2,..., T_n$, we can glue the eastern strands of $T_i$ to the western strands of $T_{i+1}$. Then we glue the eastern strands of $T_n$ to the western strands of $T_1$. See Figure~\ref{fig:LargeTangle} for an example.  
\begin{lemma}
\label{lem:twistnum}
Let $T_1, T_2,...,T_n$ be strongly alternating tangle diagrams whose Conway sum is a link diagram $D(L)$. Then $tw(D(L)) = \sum_{i=1}^n tw(T_i)$, where $tw(D(L))$ is the twist number for $D(L)$ and $tw(T_i)$ the twist number for the tangle diagram $T_i$. 
\end{lemma}
 
 \begin{proof}
First notice that taking the sum of tangles will not result in new twist regions. This is because, when taking a Conway sum, crossings that shared a twist region will still share a twist region and we introduce no new crossings. Therefore, $tw(D(L)) \leq \sum_{i=1}^n tw(T_i)$. \\
\indent
Now assume that $ tw(D(L)) < \sum_{i=1}^n tw(T_i)$. Then for some $i$, a twist region in $T_i$ and a twist region in $T_{i+1}$ become one region in $L$. This implies there exists a simple closed curve $\gamma$ which transversely intersects $D(L)$ twice in $T_i$ and twice in $T_{i+1}$. If we think back to $T_i$ lying in a unit square, then $\gamma$ must intersect the north and south edges or the east and west edges of the square. In the first case, this shows that the denominator closure is not prime, and the second, the numerator closure is not prime. But as $T_i$ is strongly alternating, this would be a contradiction, therefore $tw(D(L)) = \sum_{i=1}^n tw(T_i)$.  
 \end{proof}
 Next we consider the relationship between the twist number of a tangle diagram and the twist numbers of diagrams of its closures. 
 \begin{lemma}
 \label{lem:closuretwist}
     Let $T$ be a strongly alternating tangle diagram, and let $D(K)$ be the link diagram which comes from the numerator or denominator closure. Then:
     $$tw(T) - 2 \leq tw(D(K)) \leq tw(T).$$
 \end{lemma}
 \begin{proof}
     The upperbound is true as we are not adding crossings when closing a tangle, and hence cannot create new twist regions.\\
     \indent
     The lower bound stems from the fact that, when we choose a closure for $T$ we create two new potential bigons. If either region is a bigon, then it joins two twist regions. If both regions are bigons, the twist number is reduced by 2. In Figure~\ref{fig:twistred} we see an example of a tangle where the lower bound is sharp for both closures. 
     \begin{figure}[h]
         \centering
         \includegraphics[width=2.5in]{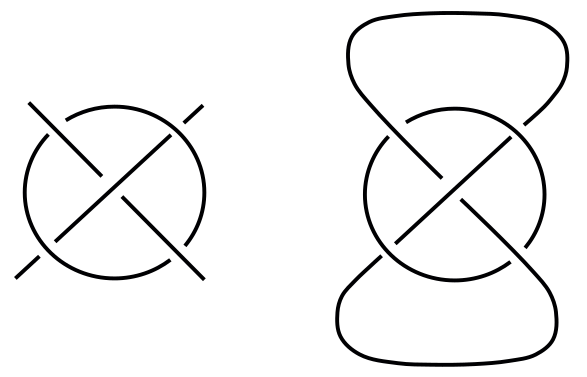}
         \caption{Left: A strongly alternating tangle with 5 twist regions. Right: The numerator closure with 3 twist regions. The denominator also results in 3 twist regions, showing the sharpness of the lower bound.}
         \label{fig:twistred}
     \end{figure}
 \end{proof}
 We will also need Theorem 3.8 from ~\cite{LEE} which we state here. This Theorem allows us to relate the cross cap numbers of the closures of strongly alternating tangles to twist numbers of their diagrams.  
 \begin{theorem}
 \label{lem:EfLee2}
  Let $L \subset S^3$ be a link of $k_L$ components with a prime, twist-reduced, alternating diagram $D(L)$. Suppose that $D(L)$ has $tw(D(L)) \geq 2$ twist regions. Let $C(L)$  denote the crosscap number of $L$. We have
 $$\left\lceil\frac{tw(D(L))}{3} \right\rceil + 2 - k_L \leq C(L) \leq tw(D(L)) + 2 - k_L$$
  Furthermore, both bounds are sharp.
 \end{theorem}

 Now that we have that the twist number is additive for strongly alternating tangles by Lemma~\ref{lem:twistnum} and can relate $C(K_i)$ and $tw(T_i)$ by Lemma~\ref{lem:closuretwist} and Theorem~\ref{lem:EfLee2}, we are ready to state Theorem~\ref{twistmain}
 
 \begin{theorem}
 \label{twistmain}
  Let $T_1$ and $T_2$ be diagrams of non-splittable, strongly alternating, twist-reduced tangles whose Conway sum is a link diagram $D(L)$. Let $C(L)$ be the crosscap number of $L$, $tw(D(L))$ be the twist number of $D(L)$ and $k_L$ be the number of link components in $L$ then
 $$\left\lceil\frac{tw(D(L))}{3} \right\rceil - k_L \leq C(L) \leq tw(D(L)) + 4 - k_L.$$
 
 \end{theorem}
 \begin{proof}
     We start with a lower bound in the proof of Lemma~\ref{lem:lowerbd} with ambiguity on $\epsilon$ where $k_1 + k_2 - \epsilon = k_L$. Then $C(K_1) + C(K_2) - 2 + \epsilon \leq C(L)$. Let $D(K_1)$ and $D(K_2)$ be the diagrams of $K_1$ and $K_2$ that arise from cutting $L$ as in Theorem~\ref{thm:CrossMain}. Notice that $tw(D(K_i)) \geq 2$ for $i=1,2$ as $T_i$ is strongly alternating, and tangle diagrams with twist number 1 will have a non prime closure. Then by Lemma~\ref{lem:EfLee2} we find for $i \in \{1,2\}$ that $\left\lceil\frac{tw(D(K_i))}{3} \right\rceil + 2 - k_i \leq C(K_i) $ where $K_i$ has $k_i$ link components. So: $$\left\lceil\frac{tw(D(K_1))}{3} \right\rceil + \left\lceil\frac{tw(D(K_2))}{3} \right\rceil + 2 + \epsilon - k_1 - k_2 \leq C(L).$$ By Lemma~\ref{lem:closuretwist} $tw(T_i) - 2 \leq tw(D(K_i))$ and from Lemma~\ref{lem:twistnum} we know $tw(T_1) + tw(T_2) = tw(D(L))$, combining the two lemmas shows  
     \begin{align}
         \left\lceil\frac{tw(D(L))}{3} \right\rceil - 2 &\leq \left\lceil\frac{tw(T_1)-2}{3} \right\rceil + \left\lceil\frac{tw(T_2)-2}{3} \right\rceil \\
         &\leq \left\lceil\frac{tw(D(K_1))}{3} \right\rceil + \left\lceil\frac{tw(D(K_2))}{3} \right\rceil.
     \end{align}
     Finally substituting in $k_1 + k_2 = k_l + \epsilon$, we find: $$\left\lceil\frac{tw(D(L))}{3} \right\rceil - k_L \leq C(L).$$

     Now we consider the upper bound. Similar to the lower bound we start with a step from Lemma~\ref{lem:upperbd}, $C(L)\leq \text{min\{} C(K_{1N}) + C(K_{2N}) + \epsilon, C(K_{1D}) + C(K_{2D}) + \epsilon\} $. By Lemma~\ref{lem:closuretwist} we see that $tw(D(K_{iN})) \leq tw(T_i)$ and $tw(D(K_{iD})) \leq tw(T_i)$ and then by Theorem~\ref{lem:EfLee2}, $$C(L) \leq tw(T_1) + tw(T_2) + 4 + \epsilon - k_1 - k_2.$$ Then substituting for $k_1 + k_2 = k_L + \epsilon$ and considering Lemma~\ref{lem:twistnum},
     $$C(L) \leq tw(D(L)) + 4 -  k_L.$$
     
     Hence, showing the claim. 
 \end{proof}
 \subsection{Jones Polynomial Bounds}
  From here we work to find bounds in terms of $T_L$, but first we have to generalize Theorem 1.6 in ~\cite{Twist} which will allow us to relate the twist number of the diagram of a link $L$ to $T_L$. Theorem 1.6 from ~\cite{Twist} only considers knots but we want a similar result for links. We start with some necessary definitions and then a generalization of Lemma 5.4 from ~\cite{Twist} which is a necessary piece of our generalization of Theorem 1.6.
  \begin{definition}
  If we let $D(L)$ be the diagram of a link $L$ we define the \textit{A resolution} and \textit{B resolution} as shown Figure~\ref{fig:resolves}. Then a \textit{Kauffman state} is a choice of resolutions for each crossing in a link diagram.
  \end{definition}
\begin{figure}[h]
    \center
    \includegraphics[width=6cm]{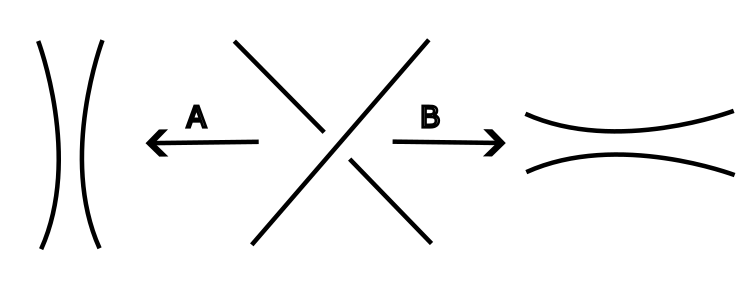}
    \caption{Given a crossing, we can resolve it to either the A or B resolution.}
    \label{fig:resolves}
\end{figure}
\begin{definition}
 Next we construct a \textit{all A (resp. all B) state graph} $G_A(D(L))$ (resp. $G_B(D(L))$) by adding edges where we performed resolutions and then contracting the simple closed curves to vertices. Let $e_A$ (resp. $e_B$) be the number of edges in $G_A(D(L))$ (resp. $G_B(D(L))$). Further, if we identify all parallel edges (edges which share two vertices) we find the \textit{reduced state graph} $G_A'(D(L))$ (resp. $G_B'(D(L))$). Let $e_A'$ (resp. $e_B')$ be the number of edges in $G_A'(D(L))$ (resp. $G_B'(D(L))$).  
\end{definition}
Now we are ready to define what it means for a link to be adequate.
\begin{definition}
We call a link diagram \textit{A-adequate} (resp. \textit{B-adequate}) if the A state (resp. B state) graph of the diagram has no one edge loops. A link diagram is called \textit{adequate} if it is both A-adequate and B-adequate, and a link is adequate if it has a diagram which is adequate. 
\end{definition}
Here we recall some terminology from ~\cite{Twist}.
  \begin{definition}
     Let $D(L)$ be the link diagram obtained by taking the Conway sum of strongly alternating tangles $T_1,...,T_n$. Let $\ell_{\text{in}}(D(L))$ denote the loss of edges in $G_A(D(L))$ and $G_B(D(L))$ as we pass from $e_A + e_B$ to $e_A' + e_B'$ which come from equivalent crossings in the same tangle $T_i$. Then let $\ell_{\text{ext}}(D(L))$ be the number of edges we lose from identification when we take the Conway sum. It follows that $\ell_{\text{in}}(D(L)) + \ell_{\text{ext}}(D(L) = e_A + e_B - e_A' - e_B'$.\\
     \indent
     For an alternating tangle diagram $T$, notice that the vertices of $G_A(T)$ and $G_B(T)$ are in 1-1 correspondence with the regions of $T$. Note that for the state graph of a tangle, if we consider the tangle lying within a disk, we have four exterior regions bounded by the disk. This means our state graphs have \textit{interior vertices} those whose region lie entirely within the interior of the disk and two \textit{exterior vertices} with corresponding region, with sides on the boundary of the disk.\\
     \indent
     For a tangle $T_i$, a \textit{bridge} of $G_A(T_i)$ or $G_B(T_i)$ is a subgraph consisting of an interior vertex $v$, and edges $e',e''$ which connect $v$ to the exterior vertices $v'$ and $v''$. We call the bridge \textit{inadmissable} if the vertices become identified in $G_A(D(L))$ or $G_B(D(L))$.
 \end{definition}

 Now we find an upper bound for $\ell_{\text{ext}}(D(L))$ in terms of the twist number. This work will largely follow the proof of Lemma 5.4 in ~\cite{Twist}. 
 \begin{lemma}
    \label{lem:lext}
     Let $T_1$ and $T_2$ be strongly alternating tangles whose Conway sum is a link diagram $D(L)$. Let $k_L$ be number of link components in $L$. Then:
     $$\ell_{\text{ext}}(D(L))  \leq \frac{tw(D(L))}{2} + k_L + 4$$
 \end{lemma}
 \begin{proof}
     For $T \in \{T_1, T_2 \}$ let $b_A(T)$, $b_B(T)$ be the number of bridges in $G_A(T)$ and $G_B(T)$ respectively. Then the contribution of $T$ to $\ell_{\text{ext}}$ will be at most $b_A(T) + b_B(T)$. Any other edge identification from moving to the reduced graph will still be counted by $\ell_{\text{int}}$.\\
     \indent
     If $b$ is a bridge there are two possibilities:
     \begin{enumerate}[i.]
         \item The edges $e', e''$ do not come from the resolutions of a single twist region. 
         \item The edges $e', e''$ come from the resolutions of a single twist region. 
     \end{enumerate}
     Notice that for type ii bridges the two crossings which result in the edges are the only two in their respective twist region. Otherwise the two edges will not be adjacent to the exterior vertices or this would no longer constitute a twist region. \\
     \indent
     For type i bridges notice that when we pass from $G_A(T)$ and $G_B(T)$ to $G_A'(T)$ and $G_B'(T)$ the contributions to $\ell_{\text{ext}}$ is half the number of twist regions involved in such bridges. Unlike in ~\cite{Twist} we can have more than one type ii bridge, as each additional type ii bridges creates a new link component. \\
     \indent 
     \textit{Case 1:} Suppose that $b_A(T) \geq 3$ or $b_B(T) \geq 3$. Without loss of generality let $b_A(T) \geq 3$, then $b_B(T) = 0$. If $b_B(T)$ were not zero then the $B$ state bridge would cross the $A$ state bridges, implying two internal vertices  which is not a bridge. \\
     \indent
     There can be any number of type ii bridges, but we notice each bridge beyond the first will add a new link component. If we have only type ii bridges $b_A(T) \leq k_T$ where $k_T$ is the number of tangle components. On the other hand if we only have type i bridges $b_A(T) \leq \frac{tw(T)}{2}$. Then for any mix of bridges we find that $b_A(T) + b_B(T) \leq \frac{tw(T)}{2} + k_T$.\\
     \indent
     \textit{Case 2:} In this case we will consider $b_A(T) = b_B(T) = 2$. Then $k$ must be at least 2, as the bridges in $G_A(T)$ and $G_B(T)$ will create a square resulting in a tangle second component. Also notice that as $G_A(T)$ has two bridges there are at least two twist regions in $T$. Then $b_A(T) + b_B(T) \leq \frac{tw(T)}{2} + k_T + 1$. \\
     \indent
     \textit{Case 3:} Either $b_A(T) \leq 2$ and $b_B(T) \leq 1$ or $b_A(T) \leq 1$ and $b_B(T) \leq 2$. Without loss of generality consider the first possibility. Then $b_A(T) + b_B(T)$ is at most three. If it's less than three we see that $k_T + 1 \geq 2$ so we need only consider when they sum to three. But as with the previous case we will have at least two twist regions as $G_A(T)$ has two bridges. Thus, $b_A(T) + b_B(T) \leq \frac{tw(T)}{2} + k_T + 1$.\\
     \indent
     Then by Lemma~\ref{lem:twistnum} we know that the twist number is additive over Conway sums. By Lemma~\ref{lem:linkcomp} $k_1 + k_2 \leq k_L + 2$. Then we find the following bound;
     $$\ell_{\text{ext}} \leq \sum_{i = 1}^2 b_A(T_i) + b_B(T_i) \leq \frac{tw(D(L))}{2} + k_L + 4.$$
   
 \end{proof}
    Now we have the tools necessary to prove the main lemma needed to find bounds for $C(L)$ in terms of $T_L$.
  \begin{lemma}
    \label{lem:TLtwister}
    Let $T_1,...,T_n$ be strongly alternating tangles whose Conway sum is a link diagram $D(L)$ for a link $L$. Then letting $\beta_L$ and $\beta'_L$ be the second and second-to-last coefficients of the Jones Polynomial of $L$, $T_L = |\beta_L| + |\beta_L'|$, and $k_L$ the number of link components, we have
    $$\frac{\text{tw}(D(L))}{2} - k_L - 2 \leq T_L \leq 2\text{tw}(D(L)). $$
 \end{lemma} 
 \begin{proof}
      We start by noting that we can mutate the link $L$ in such a way that it either is alternating or the sum of $T$ and $T'$ where $T$ is a positive 
     strongly alternating tangle and $T'$ is a negative strongly alternating tangle, without changing its Jones Polynomial~\cite{Lickorish}. Where the positive and negative refer to whether the northwest strand originates from and overcrossing or an undercrossing. In the former case we have a stronger result by Dasbach and Lin~\cite{twistJones} that $T_L = tw(L)$. \\
     \indent
     We will assume that $L$ is not alternating. Then work by Lickorish and Thistlewaite~\cite{adequate} shows that $D(L)$ is adequate. Further, by propositions 1 and 5 of ~\cite{adequate} we have $v_A + v_B = c$ where $v_A$ is the number of vertices in $G_A(D(L))$, $v_B$ the same in $G_B(D(L))$ and $c$ the number of crossings in $D(L)$. Every edge we lose when passing from $G_A(D(L))$ and $G_B(D(L))$ to $G_A(D(L))'$ and $G_B(D(L))'$ comes from either multiple edges in a twist region or an inadmissible bridge. By lemma 5.2 in ~\cite{Twist} we see that the number of edges lost due to twist regions is $c-\text{tw}(D(L))$. Work by Stoimenow shows that for an adequate link diagram (see ~\cite{twistJones} for a proof)
     \begin{align}
         T_L &= e_A' + e_B' - v_A - v_B + 2\\
         & = (e_A' + e_B' - e_A - e_B) + e_A + (e_B  - v_A - v_B) + 2\\
         & = -(c-tw(D(L)) + \ell_{ext}) + c + (c - v_A - v_B) + 2\\
         & \geq tw(D(L)) - \ell_{ext} + 2\\
         & \geq tw(D(L)) - \frac{tw(D(L))}{2} - k_L - 4 + 2 = \frac{tw(D(L))}{2} - k_L - 2.
     \end{align}
    The upper bound on $T_L$ was shown by Futer, Kalfagianni and Purcell in ~\cite{UpperBoundTwTL}.
 \end{proof}
 \begin{reptheorem}{cor:TLTwist}
      Let $T_1$ and $T_2$ be non-splittable, twist reduced, strongly alternating tangles whose Conway sum is a link $L$. If $C(L)$ is the crosscap number of $L$, $T_L = |\beta_L|+|\beta_L'|$ and $k_L$ is the number of link components in $L$ we find that,
 $$\left\lceil\frac{T_L}{6} \right\rceil - k_L \leq C(L) \leq 2T_L + k_L + 8.$$
 \end{reptheorem}
 \begin{proof}
     This follows immediately from Theorem~\ref{twistmain} and Lemma~\ref{lem:TLtwister}.
 \end{proof}
 \begin{cor}
     Let $T_1$ and $T_2$ be twist reduced, non-splittable, strongly alternating tangles whose Conway sum is a link $L$. Assume that $tw(T_i) = tw(K_iN) = tw(K_iD)$. If $C(L)$ is the crosscap number of $L$, $T_L = |\beta_L|+|\beta_L'|$ and $k_L$ is the number of link components in $L$, we have,
 $$\left\lceil\frac{T_L}{6} \right\rceil + 2 - k_L \leq C(L) \leq 2T_L + k_L + 8.$$
 \end{cor}
 \section{Generalizing to Larger Conway Sums of Tangles}
 Our goal in this section is to generalize Theorem~\ref{cor:TLTwist} to Conway sums of more than two tangles. A \textit{Conway sum} of more than 2 tangles is a closure where we connect diagrams of the tangles $T_1,T_2,...,T_l$ linearly west to east shown in Figure~\ref{fig:LargeTangle}. As with the case of the sum of two tangles, if we let $L$ be our Conway sum and $S$ a spanning surface, cutting $S$ along a Conway sphere intersecting $T_i$ will result in a spanning surface for either $K_{iN}$ or $K_{iD}$, dictating the closure for the tangle. When we cut $L$, we position $l$ Conway spheres such that $\Sigma_i$ intersects $L$ at the directional strands of $T_i$. We note that $S^3 / \bigcup_i \Sigma_i$ will not be a sphere, but we are concerned with the surfaces within the interior of each $\Sigma_i$. $S/\cup_i \Sigma_i$ will be a collection of bands and tubes which we consider in the Euler characteristic change. This section will have similar results to the previous sections but with a factor for the number of tangles. We start with the following lemma which is a generalization of Lemma~\ref{lem:linkcomp}. 
 \begin{lemma}
     \label{lem:linkcompgen}
     Let $L$ be the Conway sum of the tangles $T_1, T_2,..., T_l$, and $K_1, K_2,...,K_l$ are closures of the tangles. Then if $k_L$ is the number of link components for $L$, and $k_1, k_2,..., k_l$ the number of link components for each link respectively then $k_L = \sum_{i=1}^l k_i - l + \epsilon$ for $\epsilon=0,1,2$.
 \end{lemma}

\begin{theorem}
\label{thm:GeneralizedMain}
    Let $T_1, T_2, ... , T_l$ be non-splittable, strongly alternating tangles, and let $L$ be the Conway sum that results from the $l$ tangles. If $K_i$ is the closure of $T_i$ resulting from cutting the crosscap realizing spanning surface for $L$ for all $i \in \{1,2,...,l\}$ and $K_{iN}$ is the numerator closure of $T_i$ and $K_{iD}$ the denominator closure, then we have:
    $$\sum_{i=1}^l C(K_i) - l \leq C(L) \leq \text{min}\{\sum_{i=1}^l C(K_{iN}) + l ,\sum_{i=1}^l C(K_{iD}) +2 \}.$$
\end{theorem}

    \begin{proof}
        This proof will largely follow the work we did in Lemma~\ref{lem:lowerbd} and Lemma~\ref{lem:upperbd}. We will start by considering the upper bound. \\
        \indent
        First we consider the case where we have $K_{iN}$ for all $i \in \{1,2,....,l\}$, and spanning surfaces $S_i$ for each $K_{iN}$ such that $C(S_i) = C(K_{iN})$. Unlike in Lemma~\ref{lem:upperbd} the NW and NE strands connect to different tangles and we will find the same for the SW and SE strands. Then the spanning surface resulting from the Conway sum will have northern and southern disks attached to each of the $S_i$ by a band as seen in figure~\ref{fig:Numersurface}. Let this resulting surface be $S$.
        \begin{figure}[h]
    \center
    \includegraphics[width=8cm]{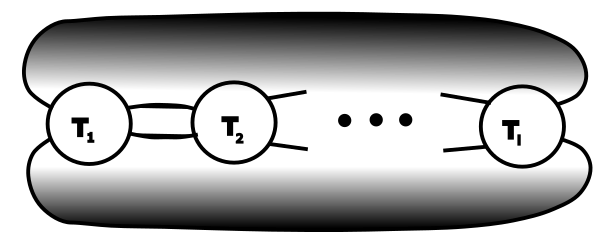}
    \caption{We see here that a surface would have to fill the shaded areas to connect the numerator closures of the tangles since the western and eastern boundaries of the $S_i$ connect to separate tangles.    }
         \label{fig:Numersurface}
        \end{figure}
        \\
        \indent  By this construction $\chi(S) = \sum_{i=1}^l \chi(S_i) - 2l + 2$ where the $2l$ comes from the bands connecting each surface to the disks, and the 2 from the disks themselves. Then by Lemma~\ref{lem:linkcompgen} we see that:
        $k_L = \sum_{i=1}^l k_{iN} - l + \epsilon$.
        Then
        $$C(S) = 2 - (\sum_{i=1}^l \chi(S_i) - 2l + 2) - (\sum_{i=1}^l k_{iN} - l + \epsilon) = \sum_{i=1}^l C(K_{iN}) + l - \epsilon.$$
        Then we see the weakest upperbound is when $\epsilon = 0$, so $$C(L) \leq C(S) = \sum_{i=1}^l C(K_{iN}) + l.$$\\
        \indent
        Meanwhile the all denominator closure case will be similar to when $l=2$. In particular $\chi(S) = \sum_{i=1}^l \chi(S_i) - l$ as we add bands to connect each of the $S_i$ to their neighboring surfaces. By Lemma~\ref{lem:linkcompgen} and similar computations to the numerator closure case we find $C(L) \leq \sum_{i=1}^l C(K_{iD}) + 2$. We have a 2 instead of an $l$ as we added half the number of bands in constructing $S$. Then we take the minimum of the denominator and numerator bounds to find an upperbound for $C(L)$. \\
        \indent
        Now we consider the lower-bound. Let $S$ be a spanning surface for $L$ such that $C(S) = C(L)$. Similar to Lemma~\ref{lem:lowerbd} we will be considering the surfaces $S_1,S_2,...,S_l$ that result from cutting along the Conway spheres $\Sigma_i$. By a similar argument to the one for Lemma~\ref{lem:curveslemma}, $S$ can be chosen so that $\Sigma_i \cap S$ contains at most one closed curve for all $i$. Notice that if any of the steps in Lemma~\ref{lem:curveslemma} were to disconnect the surface outside $\Sigma_i$, then for some other $\Sigma_j$, $i \neq j$, $T_j$ would span a disconnected surface which is a contradiction to Lemma~\ref{lem:AdamsKindred}.  Then when we cut along the Conway spheres we see that we are at most cutting along two arcs and a closed curve. \\
        \indent
        We know from Lemma~\ref{lem:linkcompgen} that $k_L = \sum_{i=1}^l k_i - l + \epsilon$ for $\epsilon = 0,1,2$. If we have closed curves along a $\Sigma_i$, when we cut we will have to add disks to both resulting boundary components which increases the Euler charactersitic by 2. For any surface resulting from cutting that is orientable we will have to add in a half twist band to make it non-orientable. Each such half twist band reduces the Euler characteristic by 1. Then $\sum_{i=1}^l \chi(S_i) = \chi(S) + t + c - b$ where $t = l$ or $t = 2l - 2$ depending on if the $K_i$ are the denominator or numerator closures, $c$ the number of closed curves on the $\Sigma_i$ which can be as large as $l$ and $b$ the number of twist bands added to make the $S_i$ non-orientable which also has maximum $l$. The value of $t$ arises from the bands which sit in $S^3 / \bigcup_i \Sigma_i$. \\
        \indent
        Now we see that 
        $$\sum_{i=1}^l (C(S_i)) =  2l - \sum_{i=1}^l \chi(S_i) - \sum_{i=1}^l k_i = 2l - \chi(S) - t - c + b - k_L - l + \epsilon.$$
        Notice that the weakest upperbound for $\sum_{i=1}^l (C(S_i))$ is when $t$ and $c$ are minimal and $b$ and $\epsilon$ are maximal. So this will be when we do not have any closed curves and each of the resulting $S_i$ need half twist bands to make them non orientable. So: $$\sum_{i=1}^l (C(S_i)) = 2l - \chi(S) - l + l  - k_L - l + 2 = C(L) + l.$$ Then moving the $l$ to the other side we see that $\sum_{i=1}^l C(K_i) - l \leq C(L)$, showing the claim.
        
        \end{proof}

As in section 3 we will now use Theorem~\ref{lem:EfLee2} to find bounds for $C(L)$ in terms of $tw(D(L))$ where $tw(D(L))$ is the twist number for a diagram of $L$.

\begin{theorem}
\label{thm:gentwist}
    Let $T_1,T_2,...,T_l$ be non-splittable, twist reduced, strongly alternating tangles and let $D(L)$ be the link diagram for the link $L$ which results from taking the Conway Sum of the tangles. Let $tw(D(L))$ denote the twist number of $D(L)$ and $C(L)$ the crosscap number for $L$, then
    $$\left\lceil\frac{tw(D(L))}{3} \right\rceil + 2 - k_L \leq C(L) \leq tw(D(L)) + l + 2 - k_L.$$
    
\end{theorem}
\begin{proof}
    We start with the bounds from Theorem~\ref{thm:GeneralizedMain}. From here we use Lemma~\ref{lem:EfLee2} and Lemma~\ref{lem:closuretwist} to get bounds on $C(L)$ with respect to the twist numbers of diagrams of $T_1,...,T_l$. Combining these two statements we find 
    $$\sum_{i=1}^l\left\lceil\frac{tw(T_i)-2}{3} \right\rceil + l - \sum_{i=1}^l k_i \leq C(L) \leq \sum_{i=1}^l tw(T_i) + 2l + 2 - \sum_{i=1}^l k_i,$$
    where $k_i$ is the number of link components for each $K_i$.   \\
    \indent
    We know that the twist number of strongly alternating tangles is additive over a Conway sum by Lemma~\ref{lem:twistnum}. So the only detail left to consider is the relationship between $k_L$ and $\sum_{i=1}^l k_i$. By Lemma~\ref{lem:linkcompgen} and a similar argument to that in Theorem~\ref{twistmain} we find the claim: 
    $$\left\lceil\frac{tw(D(L)) - 2l}{3} \right\rceil + 2 - k_L \leq c(L) \leq tw(D(L)) + l + 2 - k_L.$$
\end{proof}
The final piece of our puzzle is to find bounds in terms of $T_L$. To do this we use Lemma~\ref{lem:TLtwister} and Theorem~\ref{thm:gentwist} and the result follows. 
\begin{theorem}
    \label{thm:TLGen}
    Let $T_1,T_2,...,T_l$ be non-splittable, twist reduced, strongly alternating tangles and let $L$ be the link which results from taking the Conway Sum. Then let $C(L)$ be the crosscap number and $k_L$ the number of link components in $L$, then
    $$\left\lceil\frac{T_L-2l}{6} \right\rceil + 2 - k_L \leq C(L) \leq 2T_L + l + 6 + k_L.$$
\end{theorem}
With an additional constraint on our tangles we find the following corollary:
\begin{cor}
    \label{cor:TLGencor}
    Let $T_1,T_2,...,T_l$ be non-splittable, twist reduced, strongly alternating tangles such that $tw(T_i) = tw(D(K_{iN})) = tw(D(K_{iD}))$ for all $i \in \{1,...,l\}$. Let $L$ be the link which results from taking the Conway Sum, $C(L)$ the crosscap number, and $k_L$ the number of link components in $L$, then
    $$\left\lceil\frac{T_L}{6} \right\rceil + 2 - k_L \leq C(L) \leq 2T_L + l + 6 + k_L.$$
\end{cor}
\vspace{.2cm}
 \section{Families where $T_L$ and the Crosscap Number are Independent}
 We begin by recalling the following theorem from ~\cite{LEE} which gives linear bounds for the crosscap number of an alternating link in terms of $T_L$, where $T_L = |\beta_L| + |\beta_L'|$ and $\beta_L$ and $\beta_L'$ are the second and second to last coefficients of the Jones polynomial of $L$ respectively. 
 \begin{theorem}
 \label{thm:EfLee1}
 Let $L$ be a non-split, prime alternating link with $k$-components and with crosscap number $C(L)$. Suppose that K is not a $(2,p)$ torus link. We have
 $$\left\lceil\frac{T_L}{3} \right\rceil + 2 - k \leq C(L) \leq T_L + 2 - k$$
 where $T_L$ is as above. Furthermore, both bounds are sharp.
 \end{theorem}
 In the previous sections we showed Theorem~\ref{thm:EfLee1} generalizes to Conway sums of strongly alternating tangles. In this section we will show that Theorem~\ref{thm:EfLee1} does not generalize to arbitrary knots. 
 \begin{reptheorem}{Thm:nongen}
 
   We have the following;
    \begin{enumerate}[(a)]
        \item There exists a family of links for which $T_L \leq 2$, but $C(L)$ is arbitrarily large.
        \item There exists a family of links for which $C(L) \leq 3$, but $T_L$ is arbitrarily large. 
    \end{enumerate}
 \end{reptheorem}
 \subsection{Part (\textit{a}) of theorem~\ref{Thm:nongen}}
In this section we will consider the following family of torus knots; $T(p,q)$, where $q = j$ and $p = 2+2jk$ for odd $j>1$ and all natural numbers $k$. This family will allow us to prove part(\textit{a}) of Theorem~\ref{Thm:nongen}. We start with the following definition from Teragaito~\cite{TorusCC}. 

\begin{definition}
    \label{NPQ}

We define the value $N(p,q)$ from ~\cite{TorusCC} for fractions $\frac{p}{q}$, where $p$ and $q$ are coprime, to begin write $\frac{p}{q}$ as a continued fraction,
 
     $$\frac{p}{q} = [a_0, a_1, a_2, ... , a_n] = a_0 + \frac{1}{a_1 + \frac{1}{a_2 + \frac{1}{\cdot + \frac{1}{a_n}}}}, $$
 
 where the $a_i$ are integers, $a_0 \geq 0$, $a_i > 0 $ for $1 \leq i \leq n$, and $a_n > 1$. A continued fraction of this form is unique (cf~\cite{fractions}). Now we recursively define $b_i$ as follows: 
 $$b_0 = a_0$$ 
 \[ b_i = \begin{cases} 
      a_i & \text{if } b_{i-1} \neq a_{i-1} \text{ or if } \sum\limits_{j=0}^{i-1} b_j \text{ is odd}, \\
      0 & \text{if } b_{i-1} = a_{i-1} \text{ and } \sum\limits_{j=0}^{i-1} b_j \text{ is even}. 
   \end{cases}
\]
\end{definition}
Then, $N(p,q) = \frac{1}{2}\sum_{i=1}^n b_i$. We say a torus knot $K$ is even if the product of $p$ and $q$ is even and we say $K$ is odd otherwise. Using these definitions we can state Theorem 1.1 from~\cite{TorusCC}.
\begin{theorem}
\label{TorusCC}
 Let $K$ be the non-trivial torus knot of type $(p,q)$, where $p,q > 0$ and let $F$ be a non-orientable spanning surface of $K$ with $C(F) = C(K)$.
\begin{enumerate}[(1)]
    \item If $K$ is even, then $C(K) = N(p,q)$ and the boundary slope of $F$ is $pq$.
    \item If $K$ is odd, then $C(K) = N(pq-1,p^2)$ (resp. $N(pq+1,p^2)$) and the boundary slope of $F$ is $pq-1$ (resp. $pq+1$) if $xq \equiv -1$ (mod $p$) has an even (resp. odd) solution $x$ satisfying $0 < x < p$.
\end{enumerate}
\end{theorem}
We take advantage of (\textit{1}) from Theorem~\ref{TorusCC} to both construct our family of torus link and prove Proposition~\ref{prop:torus} below. We also need the following lemma, which gives an explicit formula for the Jones polynomial of a torus knot originally given in Proposition 11.9 of ~\cite{TorusJones}, which allows us to calculate $T_L$ for torus knots.
\begin{lemma}
\label{TorusJones}
The Jones polynomial for a torus knot $T(p,q)$ is given by;
$$V(T(p,q)) = t^{(p-1)(q-1)/2}\frac{1 - t^{p+1} - t^{q+1} + t^{p+q}}{1-t^2}.$$ 
\end{lemma}

\begin{prop}
\label{prop:torus}
Let $L = T(p,q)$ be the family of torus knots where $q > 1$ is odd and $p = 2+2qk$ for $k \in \mathbb{N}$, then $T_L \leq 2$ but $C(T(p,q))$ can be made arbitrarily large.
\end{prop}

\begin{proof}
Let $q > 1$ be odd, and $p = 2 + 2qk$ where $k$ is a natural number. To show that for all such torus knots, $L = T(p,q)$, $C(L)$ does not have a universal upper bound with respect to $T_L$, we will show that as $k$ goes to $\infty$, $C(L) $ also goes to $ \infty$, but $T_L \leq 2$. We start by computing the crosscap number of $T(p,q)$ using Theorem~\ref{TorusCC}. \\
\indent
First we notice that $\frac{p}{q} = \frac{2+2qk}{q} =  2k + \frac{1}{1 + \frac{1}{2}}$. Then $A = [2k, 1, 2]$. Then by definition~\ref{NPQ} $B = [2k, 0, 2]$. Finally, as $pq$ is even, $$C(L) = N(p,q) = \frac{2k + 0 + 2}{2} = k + 1.$$ Then as $k \rightarrow \infty$, $C(L)$ also goes to $\infty$.\\
\indent
Next by Lemma~\ref{TorusJones} we know that 
\begin{align*}
    V(L)&= t^{(p-1)(q-1)/2}\frac{1 - t^{p+1} - t^{q+1} + t^{p+q}}{1-t^2}\\
    &= t^{((2+2qk)q - (2+2qk) - q + 1)/2}\frac{1 - t^{2+2qk+1} - t^{q+1} + t^{2+2qk+q}}{1-t^2}\\
    &= t^{((2+2qk)q - (2+2qk) - q + 1)/2}(-t^{2qk+q} - t^{2qk + q - 2} - \\
    &\hspace{2cm} \dots - t^{2+2qk+1} + t^{q -1} + \dots + t^2 + 1).\\
\end{align*}
The last step arises from taking the polynomial division. Therefore, given our choices of $p$ and $q$ we see that $T_L \leq 2$. 
\end{proof}
Then Proposition~\ref{prop:torus} shows part (a) of Theorem~\ref{Thm:nongen}.

\subsection{Part (\textit{b}) of theorem~\ref{Thm:nongen}}
In this section we will work to prove part (\textit{b}) of Theorem~\ref{Thm:nongen} and the following theorem:
\begin{theorem}
    \label{thm:twob}
     There does not exist a universal linear lower bound on $C(L)$ for all links, $L$, in terms of $T_L$.
\end{theorem}

To prove this we will introduce a family of links for which $C(L)$ is uniformly bounded but $T_L$ can be made arbitrarily large. These links will be constructed by using the Whitehead double defined here:

\begin{definition}
The \textit{Whitehead double} of a knot $L$ is the satellite of the unknot clasped inside of the torus. We call it a \textit{positive} Whitehead double if the clasp is as in Figure~\ref{whitehead} and a \textit{negative} Whitehead double if not. 

\begin{figure}[h]
    
    \center
    \includegraphics[width=9cm]{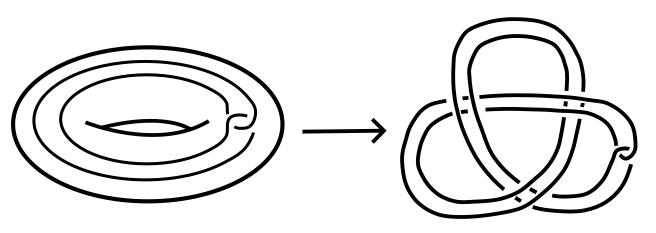}
    \caption{Here we see the unknot with a clasp contained inside the torus, then we see the resulting \textit{positive} Whitehead double with the blackboard framing when map the torus to the trefoil.}
    \label{whitehead}
    
\end{figure}

\end{definition}

The particular family is defined in this next theorem:
\begin{theorem}
   \label{lowerboundgen}
    Let $K_1,K_2,...,K_n$ be alternating knots such that $\beta'_{K_i} \neq 0$. Then we let $K$ be the connect sum of $K_1,K_2,...,K_n$ such that $K$ is alternating and $W_{-}(K)$ be the negative Whitehead double of $K$ using the blackboard framing. Then $C(L) \leq 3$ and $|\beta'_L| \geq n$. 
\end{theorem}

\begin{lemma}
\label{lem:Badequacy}
If a link is B-adequate then the negative Whitehead double of the link using the blackboard framing is also B-adequate. 
\end{lemma}
A similar statement was proven in \cite{Baker} as Proposition 7.1. They show it for the untwisted negative Whitehead double of a knot with non-negative writhe. The writhe of the knot introduces extra twists into the diagram of the untwisted Whitehead double, which can interfere with adequacy around the clasp. 
\begin{proof}
We start by showing that the blackboard $2$-cabling will be B-adequate. This is shown by Lickorish in ~\cite{Lickorish} for $n$-cablings. Let $D$ be a B-adequate diagram for our link and $D^2$ the 2 cabling. Notice in $D^2$ there will be four copies of each crossing in $D$.  Then when we have the all B-resolution state we will end up with four parallel strands instead of two as we did in $D$. If we were to have a one edge loop, then two of the strands are part of the same state circle. But these state circles are copies of the state circles for $D$ so this would contradict that $D$ is adequate. \\
\indent
Now we want to look at the negative Whitehead double of $D$ using the blackboard framing. If we let the Whitehead double be $W_{-}(D)$ we will see that $G'(W_{-}(D))$ will be the same as $G'(D^2)$ but with an additional vertex and 2 new edges as we see in Figure~\ref{posclasp}. As resolving the clasp does not create a one edge loop we see that $W_{-}(D)$ is B-adequate.
\begin{figure}[h]
    
    \center
    \includegraphics[width=9cm]{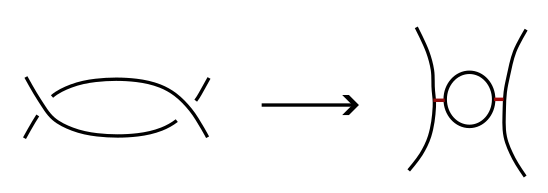}
    \caption{The result of the B-resolution on the clasp of a negative Whitehead double.}
    \label{posclasp}
    
\end{figure}
\end{proof}
Here we remind the reader that $G'_B(D(L))$ for a link diagram $D(L)$ is the reduced all B-state graph. We continue with the following lemma:
\begin{lemma}
\label{betticomp}
If $W_{-}(D(L))$ is the negative Whitehead double of a B-adequate link diagram $D(L)$ using the blackboard framing, and $b(G)$ the first Betti number for a graph, then $b(G'_B(W_{-}(D(L)))) = b(G'_B(D(L))) + 1$. 
\end{lemma}
\begin{proof}
Dasbach and Lin~\cite{HeadTail} showed in Lemma 2.5 that if $D^2$ is the two cabling of a B-adequate link diagram then $b(G_B'(D^2)) = b(G_B'(D(L)))$. For a graph $G$, $b(G) = e - v + 1$, where $e$ is the number of edges and $v$ the number of vertices. In the reduced graph when we take the two cabling every parallel copy of a state circle will also produce a new edge. Hence, the change in $v$ and $e$ will be the same between $G_B'(D^2)$ and $G_B'(D(L))$. Then when we move to $W_{-}(D(L))$ the clasp will add 2 edges and 1 vertex as we see in figure~\ref{posclasp}. Then we see that $b(G_B'(W_{-}(D(L)))) = b(G_B'(D^2)) + 1 = b(G_B'(D(L))) + 1$.
\end{proof}
 
The two previous lemmas allow us to see that the blackboard framing of the negative Whitehead double of an alternating link will be B-adequate. Also, we have a formula for the Betti number of the Whitehead double in relation to the first Betti number of the original link. The only remaining piece of the puzzle is to get from the first Betti number of the reduced B state graph to the second to last coefficient of the Jones polynomial. This comes from the following result  proven by Stoimenow in Proposition 3.1 of ~\cite{Stoimenow}.

\begin{lemma}
\label{betaprimeform}
If $D(L)$ is a B-adequate, connected diagram for a link, then in the representation of the Jones Polynomial, $V(D(L))$, we have $\alpha'_{D(L)} = \pm 1$, $\alpha_{D(L)}' \beta_{D(L)}' \leq 0$, and
$$|\beta_{D(L)}'| = e' - v' + 1 = b(G'_B(D(L))),$$
where $G_B'$ is the reduced all B state graph and $e'$ and $v'$ are the number of edges and vertices of the graph $G'_B(D(L))$, respectively. 
\end{lemma}

We now have the tools to prove Theorem~\ref{lowerboundgen}. But first we show a more specific example of a family which satisfies Theorem~\ref{Thm:nongen} part (\textit{b}). 
\begin{prop}
\label{lem:trefoil}
Let $W_{-}(K_m)$ be the negative Whitehead double using the blackboard framing of the connect sum of $m$ trefoils as in Figure~\ref{fig:trefdoub}. Then for all $m$, $C(W_{-}(K_m)) \leq 3$ and $T_{W_{-}(K_m)}$ grows with $m$. Therefore, $T_{W_{-}(K_m)}$ can be made arbitrarily large across the family of knots.

\begin{figure}[h]
    \center
    \includegraphics[width=9cm]{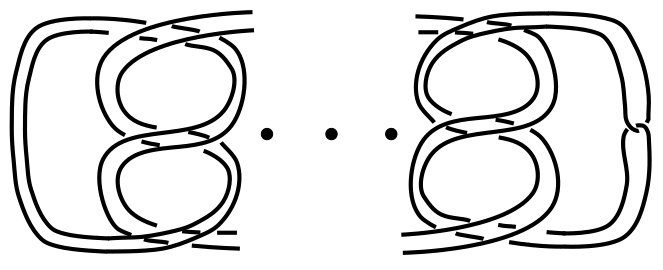}
    \caption{The negative Whitehead double of the connect sum of $m$ trefoil knots.}
    \label{fig:trefdoub}
\end{figure}
\end{prop}
\begin{proof}
The first part of the lemma is a direct result of \cite{Clark} by Clark where he shows that $c(K) \leq 2g(K) + 1$ where $g(K)$ is the genus of the knot. For any Whitehead double we can find an oriented spanning surface with genus exactly one by taking the annulus with a double twisted band at the clasp. Then $C(W_{-}(K_m)) \leq 3$ as $g(W_{-}(K_m)) = 1$. 
\\
\indent
Now we will compute $\beta'_{W_{-}(K_m)}$ by finding $G_B'(W_{-}(K_m))$. By Lemma~\ref{betaprimeform} we only need to find the number of vertices and edges as $W_{-}(K_m)$ is B-adequate. By Lemma~\ref{betaprimeform} and the graph $G_B'(W_{-}(K_m))$ shown in Figure~\ref{fig:reducedtref} 
\begin{align}
    |\beta'_{W_{-}(K_m)}| &= e(G_B'(W_{-}(K_m))) - v(G_B'(W_{-}(K_m))) + 1\\
    &= (5m + 3) - (4m + 3) + 1 = m.
\end{align}
Hence, showing that $T_{W_{-}(K_m)} \geq m$ for all $k$, proving the claim.

\begin{figure}[h]
    \center
    \includegraphics[width=9cm]{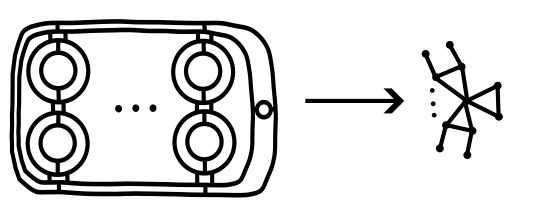}
    \caption{Left: The all B-state circle diagram. Right: The reduced state graph $G'_B(L)$. Notice the disjoint dots are not nodes for the graph but represent that we have $k$ copies of the subgraph on the left.}
    \label{fig:reducedtref}
\end{figure}
\end{proof}
Here we will introduce a more general family of knots for which Theorem~\ref{Thm:nongen} part (\textit{b}) holds true:
\begin{reptheorem}{lowerboundgen}
   Let $K_1,K_2,...,K_n$ be alternating knots such that $\beta'_{K_i} \neq 0$. Then let $K$ be the connect sum of $K_1,K_2,...,K_n$ such that $K$ is alternating, and let $W_{-}(K)$ the negative Whitehead double of $K$ using the blackboard framing. Then $C(W_{-}(K)) \leq 3$ and $|\beta'_{W_{-}(K)}| \geq n$.  
\end{reptheorem}

\begin{proof}
As in Lemma~\ref{lem:trefoil} for a Whitehead doubles such as $W_{-}(K)$, $C(W_{-}(K)) \leq 3$. Now we will work to compute $T_{W_{-}(K)}$. From ~\cite{Lickorish} we know that the Jones polynomial for $K$ will be the product of the Jones polynomials of the $K_i$. Then as all of the $K_i$ are alternating, $\alpha_{K_i}' = \pm 1$ so $\beta_K' = \sum_{i=1}^n \pm \beta_{K_i}'$. From Lemma~\ref{betaprimeform} we know that $\alpha_{K_i}' \beta_{K_i}' \leq 0$ which tells us that the signs of $\alpha_{K_i}'$ and $\beta_{K_i}'$ do not match. If we let $m$ be the number of the $\alpha'_i$ which are negative, then we see that $\beta_K' = \sum_{i=1}^n (-1)^{m\pm 1} |\beta_{K_i}'|$. Hence, in our sum the signs match so $|\beta_{K}'| = \sum_{i=1}^n |\beta_{K_i}'| $. By our hypothesis $|\beta_{K_i}'| > 0 $ for all $i$, hence $|\beta_{K}'| \geq n$.\\
\indent
 By Lemma~\ref{lem:Badequacy} we know that $W_{-}(K)$ will be B-adequate as $K$ is alternating and therefore B-adequate. Then by Lemma~\ref{betticomp} and Lemma~\ref{betaprimeform} we see that $|\beta_{W_{-}(K)}'| = |\beta_K'| + 1$ and as $|\beta_K'|$ is at least as large as the number of knots in the connect sum so is $|\beta_{W_{-}(K)}'|$ and further $T_{W_{-}(K)}$. Then $T_{W_{-}(K)}$ will grow with $n$ showing that it is unbounded across the family. 

\end{proof}
Lemma~\ref{lem:trefoil} and Theorem~\ref{lowerboundgen} both show that Theorem~\ref{Thm:nongen}(\textit{b}).

\section{Future Directions}
In Sections 2-4, we generalized the work from ~\cite{LEE} to bound the crosscap number of sums of strongly alternating tangles. Then in Section 5, introduced infinite families of knots for which their crosscap number and $T_L$ grow independently. Notice that the links we consider in Sections 2-4 are all hyperbolic, meanwhile those that we constructed in Section 5 are not hyperbolic. This leads to the following question:
\begin{ques}
    \label{ques:hyper}
    Does Theorem~\ref{thm:TLGen} generalize for all hyperbolic knots?
\end{ques}
A first step for question~\ref{ques:hyper} would be to relax the requirement that the individual tangles be strongly alternating. At the time of writing, this seems reasonable for the first step of our proof, but the uncertainty arises in moving from bounds in terms of the crosscap numbers of individual tangles to the twist number. In particular, alternating is a requirement for our usage of Theorem~\ref{lem:EfLee2}. Another potential way to move forward with this question would be to look at adequate links in general, which will be studied in future work. 


 \newpage

\bibliography{paper}        

\end{document}